\documentclass[10pt]{article}
\usepackage{amssymb}
\usepackage{graphicx}
\usepackage{xcolor} 
\usepackage{tensor}
\usepackage{fullpage} 
\usepackage{amsmath}
\usepackage{amsthm}
\usepackage{verbatim}
\usepackage{hyperref}
\usepackage{cite}
\usepackage{array}

\definecolor{green}{rgb}{0,0.8,0} 

\newtheorem{theorem}{Theorem}[section]

\newtheorem{lemma}[theorem]{Lemma}
\newtheorem{proposition}[theorem]{Proposition}

\theoremstyle{definition}

\theoremstyle{remark}
\newtheorem{remark}[theorem]{Remark}

\numberwithin{equation}{section}

\newcommand{\nnrm}[1]{{\vert\kern-0.25ex\vert\kern-0.25ex\vert #1 
		\vert\kern-0.25ex\vert\kern-0.25ex\vert}}

\newcommand{\rd}{\partial}
\newcommand{\nb}{\nabla}

\newcommand{\calC}{\mathcal C}

\newcommand{\R}{\mathbb{R}}

\newcommand{\bbn}{\mathbb N}

\newcommand{\bq}{\begin{equation}}
\newcommand{\eq}{\end{equation}}
\newcommand{\e}{\varepsilon}
\newcommand{\lt}{\left}
\newcommand{\rt}{\right}

\newcommand{\pa}{\partial}

\newcommand{\intr}{\int_{\R^3}}
\newcommand{\intrr}{\iint_{\R^3 \times \R^3}}

\begin{document}
\bibliographystyle{plain}
 \title{Global-in-time existence of weak solutions for Vlasov--Manev--Fokker--Planck system}
\author{Young-Pil Choi\thanks{Department of Mathematics, Yonsei University, Seoul 03722, Republic of Korea. E-mail: ypchoi@yonsei.ac.kr} \and In-Jee Jeong\thanks{Department of Mathematics and RIM, Seoul National University, Seoul 08826. E-mail: injee\_j@snu.ac.kr}   }

\date\today
 
  \maketitle

\renewcommand{\thefootnote}{\fnsymbol{footnote}}
\renewcommand{\thefootnote}{\arabic{footnote}}

\begin{abstract} 
We consider the Vlasov--Manev--Fokker--Planck (VMFP) system in three dimensions, which differs from the Vlasov--Poisson--Fokker--Planck in that it has the gravitational potential of the form $-1/r - 1/r^2$ instead of the Newtonian one. For the VMFP system, we establish the global-in-time existence of weak solutions. The proof extends to several related kinetic systems. 
\end{abstract}

\section{Introduction}
In this paper, we establish the global-in-time existence of weak solutions to the gravitational \textit{Vlasov--Manev--Fokker--Planck} (VMFP) system:
\begin{equation} \label{eq:VR}
\left\{
\begin{aligned}
&\rd_t f + v\cdot\nb_x f +   \nb K\star \rho \cdot \nb_v f = \sigma\nabla_v \cdot ( \nabla_v f + vf), \quad (x,v) \in \R^3 \times \R^3, \quad t > 0,\\
& \rho = \intr f \,dv,
\end{aligned}
\right.
\end{equation}
where the interaction potential $K:\R^3 \to \R$ is given by 
\bq\label{Manev}
K(x) = K_{\rm M}(x) + K_{\rm C}(x) := \frac{C_{\rm M}}{|x|^2} + \frac{C_{\rm C}}{|x|}
\eq
with some positive constants $C_{\rm M}$ and $C_{\rm C}$. Here $K_{\rm M}$ is called the {\it pure Manev} potential and $K_{\rm C}$ is the attractive Coulomb potential. Note that the force field $K \star \rho$ can be rewritten as a linear combination of $(-\Delta)^{-1} \rho$ and $(-\Delta)^{-\frac12} \rho$. The above form of $K$ has been proposed by Manev in \cite{Man1, Man2, Man4, Man3} as a modification of the Newtonian gravitational law, to explain various phenomena observed in planetary systems. We refer to the works \cite{CJpre,BDIV97,IVDB98} for an extensive list of literature on this modification of Newtonian gravitational law. As we shall demonstrate below, the Manev potential is ``critical'' in several ways and raises very interesting mathematical challenges. Already at this point, let us note that $\nb K_M \star \rho$ has the same regularity in $x$ with $\rho$ itself. 

In the absence of the pure Manev potential and the linear Fokker--Planck term, i.e. $C_{\rm M} = 0$ and $\sigma=0$, the VMFP system \eqref{eq:VR} reduces to the well-known gravitational Vlasov--Poisson system, which has been extensively studied. The global-in-time existence of weak sand classical solutions are shown in \cite{Ars75, BD85, HH84, LP91, Pal12, Pfa92, UO78}. For the Vlasov--Poisson--Fokker--Planck system, the system \eqref{eq:VR} with $C_{\rm M} = 0$, the global existence of weak solutions, long time asymptotics, and classical solutions near Maxwellian are also investigated in \cite{BD95, Bou93,  CCJ21, CCS19, CS95, CSV95, Degond86, HH84, HJ13, Vic91, VO90}. 

Despite these fruitful developments on the existence theory for the Vlasov--Poisson system or Vlasov--Poisson--Fokker--Planck system, due to the difficulties created by the strong singularity in the interaction potential, there is limited literature on the existence theory for the case with the Manev potential. The local-in-time well-posedness for the Vlasov--Manev system, i.e. \eqref{eq:VR} with $\sigma=0$, is established in \cite{IVDB98} and the finite-time singularity formation for the system \eqref{eq:VR} in the case $C_{\rm C} =0$ and $\sigma = 0$ is given in  \cite{BDIV97}. Proving singularity formation in the Manev case ($C_{\rm C}, C_{\rm M} >0$) seems to be a challenging open problem. In our recent work \cite{CJpre}, we obtained the local-in-time well-posedness theory in the case of potentials even more singular than that of Manev, namely $K$ of the form $|x|^{-\alpha}$ with $\alpha \geq d-1$ where $d$ is the dimension of the space. In the three dimensional case $d=3$, this local  well-posedness theory covers the regime $\alpha \in [2, \frac94)$. Furthermore, the finite-time singularity formation result in \cite{CJpre} covers the case $\sigma > 0$ when $\alpha > 2$. To the best of our knowledge, the only existing existence theory for the Vlasov--Manev-type system is local in time \cite{BDIV97, CJpre}. In the current work, we obtain for the first time global solutions of the VMFP system \eqref{eq:VR}. 
 
\begin{theorem}\label{main_thm} Suppose that the initial data $f_0$ satisfies 
\[
f_0 \in (L^1_+ \cap L^\infty)(\R^3 \times \R^3), \quad (|x|^2 + |v|^2 + {\bf 1}_{\sigma >0}|\log f_0|)f_0 \in L^1(\R^3 \times \R^3),
\] 
\[
(K\star\rho_0)f_0 \in L^1(\R^3 \times \R^3), 
\] and \begin{equation*}
	\begin{split}
		C_{\rm M}\|f_0\|_{L^1}^{\frac13} < \varepsilon_0
	\end{split}
\end{equation*} 
where $\varepsilon_0>0$ is a small absolute constant. Then, for any $\sigma, C_C\ge0$, there exists a global weak solution of the equation \eqref{eq:VR}--\eqref{Manev} satisfying
\[
f \in \calC([0,T];L^1(\R^3 \times \R^3)) \cap L^\infty(\R^3 \times \R^3 \times (0,T)) \quad \] and
\[ 
(K \star \rho) f \in L^\infty(0,T;L^1(\R^3 \times \R^3))
\] for all $T>0$, with initial data $f_0$. 
\end{theorem}

\begin{remark} As stated above, our existence theory covers the case $\sigma = 0$, and in this case the assumption on the entropy $f_0 \log f_0$ is not required.
\end{remark}

\begin{remark} Let us give some comments on the smallness assumption on $C_{\rm M}\|f_0\|_{L^1}^{\frac13}$.
\begin{itemize}
\item[(i)] As mentioned above, the nonglobal existence of smooth solutions to the Vlasov--(pure)Manev system, i.e. \eqref{eq:VR} with $C_{\rm C} =0$ and $\sigma = 0$ is shown in \cite{BDIV97} (see also \cite{CJpre}) under certain assumptions on the initial data. More precisely, if the initial data $f_0$ satisfies
\bq\label{blow_asp}
\intrr |v|^2 f_0(x,v)\,dxdv <  \intr (K_M \star \rho_0) \rho_0\,dx = C_{\rm M} \intrr \frac{1}{|x-y|^2}\rho_0(x) \rho_0(y)\,dxdy,
\eq
then the solution $f$ to the system \eqref{eq:VR} with $C_{\rm C} =0$ and $\sigma = 0$ cannot exist globally in time. The mechanism of blow-up is strong attraction between particles, which suggests that gravitational collapse occurs in finite time (the actual proof is based on a contradiction argument, though). This makes it likely that even weak solutions do not exist after blowup of smooth solutions. On the other hand, one may check that smallness assumption on the initial mass makes it impossible to satisfy \eqref{blow_asp}. In this regard, the smallness assumption $C_{\rm M}\|f_0\|_{L^1}^{\frac13}$ is necessary. 
\item[(ii)] In several respects, the behavior of Vlasov--Manev system in $d=3$ resembles that of attractive Vlasov--Poisson system in $d=4$. In the latter, it is well known that smooth solutions blow up for ``large'' data but remains globally bounded for ``small'' data, see \cite{Ho1,Ho2,HH84}. 
\item[(iii)] If the singularity of the kernel is mildly singular relative to the pure Manev case, then the smallness assumption can be removed. More precisely, if we consider the interaction potential $K = K(x)$ given by
 \bq\label{k_subM}
 K(x) = \frac1{|x|^\alpha} \quad \mbox{with } \alpha \in (0,2),
 \eq
 then we have the global-in-time existence of weak solution in Theorem \ref{main_thm}  without any smallness assumptions on the initial data. For more details, we refer to Remark \ref{rmk_mild}. We also would like to mention that if $K$ is given by \eqref{k_subM}, the  argument for singularity formation in \cite{BDIV97} cannot be applied. In this regard, the Manev potential case is critical. 
\item[(iv)] If we consider the repulsive interaction potential, for instance, in the case where $K$ is replaced with 
\bq\label{k_rep}
K_{\rm rep}(x) = -\frac1{|x|^\alpha} \quad \mbox{with } \alpha \in (0,2],
\eq
then the global-in-time existence of weak solution for such system can be obtained without any smallness assumptions on the initial data. See Remark \ref{rmk_rep} for details. 
\item[(v)] In the two dimensional case, the pure Manev potential is given by $K(x) = |x|^{-1}$. In this case, any smallness assumption on the initial data is not required for the global existence of weak solutions. We give some details in Remark \ref{rmk_2D}. Obviously, the same result holds for the repulsive interaction case, i.e., $K(x) = -|x|^{-1}$, in two dimensions.
\end{itemize}
\end{remark}


The proof of Theorem \ref{main_thm} is given in the next section. The main difficulty is the singular and attractive interactional potential $K_{\rm M}$. We first mollify the singular interaction potential $K$ by introducing regularization parameter $\e$ and consider a solution $f^{(\e)}$ to that regularized system. We then show some uniform-in-$\e$ estimates on the kinetic energy, in which the smallness assumption on $C_{\rm M}\|f_0\|_{L^1}^{\frac13}$ is required (Proposition \ref{prop_reg}). In order to pass to the limit $\e \to 0$, we use the velocity averaging lemma (Lemma \ref{lem_va}) to have strong compactness. This together with the Calder\'on--Zygmund lemma and Hardy--Littlewood--Sobolev inequality, we show that the limit function $f:= \lim_{\e \to 0} f^{(\e)}$ satisfies the system \eqref{eq:VR} in the sense of Theorem \ref{main_thm}.

 %
 %
 %
 %
 %
 %
\section{Global existence of weak solutions}

\subsection{Regularized VMFP system}

For the existence of weak solutions to \eqref{eq:VR}, we first regularize the system as follows: 
\begin{align}\label{eq:reg}
\begin{aligned}
&\pa_t f^{(\e)} + v\cdot\nabla_x f^{(\e)}  +  \nabla K^\e \star \rho^{(\e)}\cdot \nabla_v f^{(\e)}   = \sigma\nabla_v \cdot ( \nabla_v f^{(\e)} + vf^{(\e)}),
\end{aligned}
\end{align}
subject to the initial data
\[
f_0^{(\e)} = f^{(\e)}(x,v,0) := f_0(x,v)  \mathbf{1}_{\{|v|\le 1/\e\}},
\]
where 
\[
\rho^{(\e)} := \int_{\R^3} f^{(\e)} \,dv, \quad \rho^{(\e)} u^{(\e)} := \int_{\R^3} vf^{(\e)} \,dv, 
\]
and $K^\e = K^\e(x)$ is given as
\[
K^\e(x) = K^\e_{\rm M}(x) + K^\e_{\rm C}(x):= \frac{C_{\rm M}}{(\e + |x|^2)} + \frac{C_{\rm C}}{(\e + |x|^2)^{\frac12}}.
\]
Note that $K^\e_{\rm M}(x) \leq K_{\rm M}(x)$ and $K^\e_{\rm C}(x) \leq K_{\rm C}(x)$ for all $x \in \R^3$, thus the regularized potential $K^\e(x)$ satisfies $K^\e(x) \leq K(x)$ for all $x\in \R^3$.

We notice that the global-in-time existence of weak solutions to the regularized system \eqref{eq:reg} follows by the standard existence theory for kinetic equations since the force field $\nabla K^\e \star \rho^{(\e)}$ is bounded and Lipschitz continuous.

\subsection{Uniform-in-$\e$ estimates}\label{ssec:uniform}

We begin with the $L^\infty$ bound estimate.
\begin{lemma}\label{lem_inf} Let $T>0$, $p \in [1,\infty]$, and $f^{(\e)}$ be the weak solution to \eqref{eq:reg} on the interval $[0,T]$. Then we have
\[
\sup_{0 \leq t \leq T}\|f^{(\e)}(\cdot,\cdot,t)\|_{L^p} \le \|f_0^{(\e)}\|_{L^p}e^{3\sigma(1-1/p) T} \leq \|f_0\|_{L^p}e^{3\sigma(1-1/p) T} .
\]
\end{lemma}
\begin{proof}
We estimate
\begin{align*}
\frac{d}{dt}\intrr (f^{(\e)})^p\,dxdv 
&= - (p-1) \intrr (f^{(\e)})^p \nabla_v \cdot \lt(   \nabla K^\e \star \rho^{(\e)}  \rt) dxdv\cr
&\quad - \sigma p(p-1) \intrr (f^{(\e)})^{p-2} |\nabla_v f^{(\e)}|^2\,dxdv +3\sigma(p-1)\intrr (f^{(\e)})^p\,dxdv \cr
&= - \frac{4\sigma (p-1)}{p}\intrr |\nabla_v (f^{(\e)})^{p/2}|^2\,dxdv +3\sigma(p-1)\intrr (f^{(\e)})^p\,dxdv
\end{align*}
for $p \in [1,\infty)$. This implies
\[
\|f^{(\e)}(\cdot,\cdot,t)\|_{L^p}^p + \frac{4\sigma(p-1)}{p}\int_0^t  \|\nabla_v (f^{(\e)})^{p/2}(\cdot,\cdot,s)\|_{L^2}^2\,ds = \|f^{(\e)}_0\|_{L^p}^p e^{3\sigma(p-1)t}.
\]
Hence we have the desired result.
\end{proof}

We next provide an auxiliary lemma showing some relationship between the local density and the kinetic energy. This lemma will be used to estimate the interaction energy. 
\begin{lemma}\label{lem_tech} Suppose $f \in L^1 \cap L^\infty(\R^3 \times \R^3)$ and $|v|^2 f \in L^1(\R^3 \times \R^3)$. Then there exists a constant $C>0$ such that
\[
\|\rho\|_{L^\frac53} \leq (C\|f\|_{L^\infty} + 1)\lt(\intrr |v|^2 f\,dxdv\rt)^{\frac35}.
\]
In particular, we have
\begin{align*}
\|\rho\|_{L^p} & \leq (C\|f\|_{L^\infty}^\alpha + 1)\lt(\intrr |v|^2 f\,dxdv\rt)^{\frac35 \alpha}\|\rho\|_{L^1}^\alpha \quad \mbox{for all} \quad p \in \lt[1,\, \frac53\rt],
\end{align*}
where $\rho = \intr f\,dv$ and $\alpha = \frac52 \lt(1 - \frac1p \rt)$. 
\end{lemma}
\begin{proof} Note that for any $R>0$
\[
\rho = \intr f\,dv = \lt(\int_{|v| \geq R} + \int_{|v| \leq R} \rt) f\,dv \leq \frac1{R^2} \intr |v|^2 f\,dv + C\|f\|_{L^\infty} R^3.
\]
We now take $R = \lt(\intr |v|^2 f\,dv \rt)^{1/5}$ to obtain
\[
\rho \leq (C\|f\|_{L^\infty} + 1)\lt(\intr |v|^2 f\,dv \rt)^\frac35.
\]
This implies
\[
\|\rho\|_{L^\frac53} \leq (C\|f\|_{L^\infty} + 1)\lt(\intrr |v|^2 f\,dxdv\rt)^{\frac35}.
\]
Since $\rho \in L^1(\R^3)$, the $L^p$ interpolation inequality gives the desired result.
\end{proof}
We then show the bound estimate on the interaction energy in the lemma below.
\begin{lemma}\label{lem_tech2} Let $T>0$ and $f^{(\e)}$ be the weak solution to \eqref{eq:reg} on the interval $[0,T]$. Then we have
\[
\lt| \intrr K^\e(x-y)\rho^{(\e)} (x) \rho^{(\e)} (y) \,dxdy\rt| \leq CC_{\rm M}\|\rho_0\|_{L^1}^{\frac13}\|\rho^{(\e)}\|_{L^\frac53}^{\frac53} + CC_{\rm C}\|\rho_0\|_{L^1}^{\frac76}\|\rho^{(\e)}\|_{L^\frac53}^{\frac56}
\]
where $C>0$ is independent of $\e>0$.
\end{lemma}
\begin{proof} 
We first recall the classical Hardy--Littlewood--Sobolev inequality:
\bq\label{HLS}
\lt|\intrr \mu(x)|x-y|^{-\lambda} \nu(y)\,dxdy\rt| \leq C_{p,\lambda} \|\mu\|_{L^p} \|\nu\|_{L^q}
\eq
for $\mu \in L^p(\R^3)$, $\nu \in L^q(\R^3)$, $1 < p,\,q < \infty$, $1/p + 1/q + \lambda/3 = 2$, and $0 < \lambda < 3$. 

By $L^p$-interpolation inequality, for $1 \leq p \leq \gamma$, we observe
\[
\|\rho\|_{L^p} \leq \|\rho\|_{L^1}^{1-\alpha} \|\rho\|_{L^\gamma}^\alpha, \quad \frac1p = 1 - \alpha + \frac\alpha\gamma
\]
and
\[
\|\rho\|_{L^q} \leq \|\rho\|_{L^1}^{1-\beta} \|\rho\|_{L^\gamma}^\beta, \quad \frac1q = 1 - \beta + \frac\beta\gamma.
\]
This gives
\[
\|\rho\|_{L^p}\|\rho\|_{L^q} \leq \|\rho\|_{L^1}^{2-(\alpha + \beta)} \|\rho\|_{L^\gamma}^{\alpha + \beta}.
\]
Here if $p$ and $q$ satisfy $1/p + 1/q + \lambda/3 = 2$, then we readily check 
\[
\alpha + \beta = \frac{\gamma}{\gamma - 1} \cdot \frac{\lambda}{3}.
\]
This together with taking $\gamma = \frac53$ and $\lambda = 2$ yields
\[
 \intrr K^\e_{\rm M}(x-y)\rho^{(\e)} (x) \rho^{(\e)} (y) \,dxdy \leq \intrr K_{\rm M}(x-y)\rho^{(\e)} (x) \rho^{(\e)} (y) \,dxdy \leq CC_{\rm M}\|\rho_0\|_{L^1}^{\frac13}\|\rho^{(\e)}\|_{L^\frac53}^{\frac53}
\]
for some $C>0$ independent of $\e > 0$. Analogously, we  find
\[
 \intrr K^\e_{\rm C}(x-y)\rho^{(\e)} (x) \rho^{(\e)} (y) \,dxdy  \leq CC_{\rm C}\|\rho_0\|_{L^1}^{\frac76}\|\rho^{(\e)}\|_{L^\frac53}^{\frac56},
\]
for some $C>0$ independent of $\e > 0$. This completes the proof. 
\end{proof}

As we have mentioned in the introduction, we will need some compactness results for the local density $\rho^{(\e)}$. For this, in the following proposition, we provide the uniform-in-$\e$ estimate on the second moment of $f^{(\e)}$. 
\begin{proposition}\label{prop_reg}
Let $T>0$ and $f^{(\e)}$ be the weak solution to \eqref{eq:reg} on the interval $[0,T]$. Suppose that $C_{\rm M}\|\rho_0\|_{L^1}^{\frac13}$ is small enough. Then the following bound estimate on the second moment holds
\bq\label{uni_est}
\intrr \left(\frac{|v|^2}{2} + \frac{|x|^2}{2} \right) f^{(\e)} \,dxdv + \sigma \int_0^t\intrr\frac{1}{f^{(\e)}} |\nabla_v f^{(\e)} + vf^{(\e)}|^2  \,dxdv ds\le C,
\eq
for all $t \in [0,T]$ and some $C > 0$ independent of $\e$.
 
\end{proposition}

\begin{proof}[Proof of Proposition \ref{prop_reg}]
Straightforward computations give
\[
\frac12\frac{d}{dt}\left(\intrr |v|^2 f^{(\e)} \,dxdv\right) = \intrr( \nabla K^\e \star \rho^{(\e)})\cdot v f^{(\e)} \,dxdv  - \sigma\intrr v \cdot (\nabla_v f^{(\e)} + vf^{(\e)}) \,dxdv.
\]
On the other hand, we get
\begin{align*}
\frac{d}{dt}\left(\frac{1}{2}\intrr K^\e(x-y)\rho^{(\e)}(x) \rho^{(\e)}(y) \,dxdy\right) 
&= \intrr K^\e(x-y)\partial_t\rho^{(\e)}(x) \rho^{(\e)}(y) \,dxdy\\
&= -\intrr K^\e(x-y)\nabla \cdot (\rho^{(\e)} u^{(\e)})(x) \rho^{(\e)}(y)\,dxdy\\
&= \intrr \nabla K^\e(x-y) \rho^{(\e)}(y) \cdot (\rho^{(\e)} u^{(\e)})(x) \,dxdy\\
&= \intr (\nabla K^\e \star \rho^{(\e)}) \cdot  (\rho^{(\e)} u^{(\e)}) \,dx\\
&= \intrr(\nabla K^\e \star \rho^{(\e)}) \cdot v f^{(\e)} \,dxdv.
\end{align*}
This yields
\begin{align*}
&\frac{d}{dt}\left(\frac12\intrr |v|^2 f^{(\e)} \,dxdv - \frac{1}{2}\intrr K^\e(x-y)\rho^{(\e)}(x) \rho^{(\e)}(y) \,dxdy\right)\cr
&\quad = - \sigma\intrr v \cdot (\nabla_v f^{(\e)} + vf^{(\e)}) \,dxdv.
\end{align*}
We then combine the previous estimates with the following entropy estimate
\begin{align*}
&\frac{d}{dt}\left(\intrr  f^{(\e)} \log f^{(\e)} \,dxdv\right)\cr
&\quad = \intrr (\partial_t f^{(\e)}) \log f^{(\e)}\,dxdv\\
&\quad = - \intrr (   \nabla K^\e \star \rho^{(\e)})\cdot \nabla_v f^{(\e)} \,dxdv -\sigma \intrr \frac{\nabla_v f^{(\e)}}{f^{(\e)}} \cdot (\nabla_v f^{(\e)} + v f^{(\e)})\,dxdv\\
&\quad =   -\sigma \intrr \frac{\nabla_v f^{(\e)}}{f^{(\e)}} \cdot (\nabla_v f^{(\e)} + v f^{(\e)})\,dxdv,
\end{align*}
and thus
\begin{align}\label{energy1}
\begin{aligned}
\frac{d}{dt}&\left( \intrr \left(\frac{|v|^2}{2}   +   \log f^{(\e)} \right) f^{(\e)} \,dxdv - \frac{1}{2} \intrr K^\e(x-y)\rho^{(\e)} (x) \rho^{(\e)} (y) \,dxdy\right)\\
&\qquad + \sigma\intrr\frac{1}{f^{(\e)}} |\nabla_v f^{(\e)} + vf^{(\e)}|^2  \,dxdv = 0.
\end{aligned}
\end{align}
We notice that the entropy can be negative, and therefore we need to estimate the negative part of the entropy. For this, we first estimate the second moment on the spatial variable as
\[
\frac{d}{dt}\left(\intrr \frac{|x|^2}{2} f^{(\e)} \,dxdv\right) = \intrr x \cdot v f^{(\e)} \,dxdv \le \intrr \left(\frac{|x|^2}{2} + \frac{|v|^2}{2}\right) f^{(\e)} \,dxdv.
\]
Moreover, we observe that the following inequality holds
\[
2\intrr f^\varepsilon \log_{-}f^\varepsilon \,dx dv  \le \intrr f^\varepsilon \left( \frac{|x|^2}{2} + \frac{|v|^2}{2} \right) dx dv  + \frac{1}{e}\intrr e^{-\frac{|v|^2}{4} - \frac{|x|^2}{4}} dx dv,
\]
where $\log_{-}g(x) := \max\{0, -\log g(x)\}$. We then combine these estimates with \eqref{energy1} to get
\begin{align*}
&\intrr \left(\frac{|v|^2}{2} + \frac{|x|^2}{2} +  |\log f^{(\e)}| \right) f^{(\e)} \,dxdv - \frac{1}{2} \intrr K^\e(x-y)\rho^{(\e)} (x) \rho^{(\e)} (y) \,dxdy\\
& \quad+\sigma \int_0^t\intrr\frac{1}{f^{(\e)}} |\nabla_v f^{(\e)} + vf^{(\e)}|^2  \,dxdv ds \\
&\qquad \le \intrr\left(\frac{|v|^2}{2} + \frac{|x|^2}{2} +  |\log f_0^{(\e)}| \right) f_0^{(\e)} \,dxdv  - \frac{1}{2} \intrr K^\e(x-y)\rho^{(\e)}_0(x) \rho^{(\e)}_0 (y) \,dxdy \\
&\qquad \quad   +\int_0^t\intrr \left(|v|^2 +|x|^2\right)f^{(\e)}\,dxdvds + C,
\end{align*}
where $C$ is a positive constant independent of $\e$.

We now further need to estimate the interaction energy. We first combine Lemmas \ref{lem_inf} and \ref{lem_tech} to obtain
\[
\|\rho^{(\e)}\|_{L^\frac53} \leq (C\|f^{(\e)}\|_{L^\infty} + 1)\lt(\intrr |v|^2 f^{(\e)}\,dxdv\rt)^{\frac35} \leq C\lt(\intrr |v|^2 f^{(\e)}\,dxdv\rt)^{\frac35} 
\]
for some $C>0$ independent of $\e>0$. This together with Lemma \ref{lem_tech2} yields
\begin{align*}
&\lt| \intrr K^\e(x-y)\rho^{(\e)} (x) \rho^{(\e)} (y) \,dxdy\rt| \cr
&\quad \leq CC_{\rm M}\|\rho_0\|_{L^1}^{\frac13}\lt(\intrr |v|^2 f^{(\e)}\,dxdv\rt)^{\frac35 \cdot \frac53}  + CC_{\rm C}\|\rho_0\|_{L^1}^{\frac76}\lt(\intrr |v|^2 f^{(\e)}\,dxdv\rt)^{\frac35 \cdot \frac56}\cr
&\quad  = CC_{\rm M}\|\rho_0\|_{L^1}^{\frac13}\intrr |v|^2 f^{(\e)}\,dxdv  + CC_{\rm C}\|\rho_0\|_{L^1}^{\frac76}\lt(\intrr |v|^2 f^{(\e)}\,dxdv\rt)^{\frac12}.
\end{align*}
We then use the smallness assumption on $C_{\rm M}\|\rho_0\|_{L^1}^{\frac13}$ and Young's inequality to have
\begin{align*}
&\intrr \left(\frac{|v|^2}{2} + \frac{|x|^2}{2} +  |\log f^{(\e)}| \right) f^{(\e)} \,dxdv +\sigma \int_0^t\intrr\frac{1}{f^{(\e)}} |\nabla_v f^{(\e)} + vf^{(\e)}|^2  \,dxdv ds \\
&\quad \le \intrr\left(\frac{|v|^2}{2} + \frac{|x|^2}{2} +   |\log f_0^{(\e)}| \right) f_0^{(\e)} \,dxdv + 2\int_0^t\intrr \left(|v|^2 +|x|^2\right)f^{(\e)}\,dxdvds  + C
\end{align*}
for some $C>0$ independent of $\e>0$. Finally, applying Gr\"onwall's lemma to the above concludes the desired result.
\end{proof}

 { 
 \begin{remark}\label{rmk_mild} If the interaction potential $K = K(x)$ is given by \eqref{k_subM},  then we use a similar argument as in the proof of Lemma \ref{lem_tech2} to find
 \[
\lt| \intrr K(x-y) \rho(x)\rho(y)\,dxdy\rt| \leq C\|\rho_0\|_{L^1}^{2 - \frac56 \alpha}\|\rho\|_{L^\frac53}^{\frac56 \alpha}.
 \]
On the other hand, by Lemma \ref{lem_tech2}, we get
\[
\|\rho\|_{L^\frac53}^{\frac56 \alpha} \leq C\lt(\intrr |v|^2 f \,dxdv\rt)^{\frac35 \cdot \frac56 \alpha} = C\lt(\intrr |v|^2 f \,dxdv\rt)^{\frac\alpha2}, 
\]
and thus
 \[
\lt| \intrr K(x-y) \rho (x)\rho(y)\,dxdy\rt| \leq C\|\rho_0\|_{L^1}^{2 - \frac56 \alpha}\lt(\intrr |v|^2 f\,dxdv\rt)^{\frac\alpha2}.
 \]
 Since $\frac\alpha2 < 1$, we use the Young's inequality to deduce that for any $\delta > 0$
  \[
\lt| \intrr K(x-y) \rho(x)\rho(y)\,dxdy\rt| \leq \delta\intrr |v|^2 f\,dxdv + C,
 \]
where $C>0$ is independent of the solution. This shows that if the singularity of interaction potential is given as in \eqref{k_subM}, i.e. less singular than the pure Manev case, then (by regularizing the kernel as before) the uniform-in-$\e$ bound estimates \eqref{uni_est} can be obtained without any smallness assumption on the initial data.
 \end{remark}
 }

\begin{remark}\label{rmk_rep} If the interaction potential is given by \eqref{k_rep}, then we do not need to estimate the interaction energy since it is nonnegative. In particular, any smallness assumption on the initial data is not required. 
\end{remark}
 
\begin{remark}\label{rmk_2D}
In the two dimensional case, one can easily check that 
\[
\|\rho\|_{L^2} \leq (C\|f\|_{L^\infty} + 1)\lt(\intrr |v|^2 f\,dxdv\rt)^{\frac12}.
\]
under the same assumptions as in Lemma \ref{lem_tech}. If we consider the pure Manev potential, i.e., $K = |x|^{-1}$, then applying almost the same argument as in the proof of Lemma \ref{lem_tech2} gives
\bq\label{HLS}
\lt|\int_{\R^2 \times \R^2} |x-y|^{-1} \rho(x)\rho(y)\,dxdy\rt| \leq C \|\rho_0\|_{L^1} \|\rho\|_{L^2}.
\eq
Thus, the interaction energy can be estimated as
\[
\lt|\int_{\R^2 \times \R^2} |x-y|^{-1} \rho(x)\rho(y)\,dxdy\rt| \leq C\lt(\intrr |v|^2 f\,dxdv\rt)^{\frac12} \leq \delta\intrr |v|^2 f^{(\e)}\,dxdv + C
\]
for any $\delta >0$, with an absolute constant $C > 0$. Then, for the same reason as in Remark \ref{rmk_mild}, we have the uniform bound estimate \eqref{uni_est} without any smallness assumption on the initial data.
\end{remark}

\subsection{Proof of Theorem \ref{main_thm}}
Now, we provide the global existence of the weak solution to \eqref{eq:VR} based on a compactness argument. For this purpose, we need the following modified velocity averaging lemma obtained in \cite[Lemma 2.7]{KMT13}, see also \cite{Glassey96, PS98}.
\begin{lemma}\label{lem_va}
Let $\{f^m\}$ be bounded in $L_{loc}^p(\R^3 \times \R^3 \times [0,T])$ with $1 < p < \infty$ and $\{G^m\}$ be bounded in $L_{loc}^p(\R^3 \times \R^3 \times [0,T])$. If $f^m$ and $G^m$ satisfy
\[
\pa_t f^m + v \cdot \nabla_x f^m = \nabla_v^\ell G^m, \quad f^m|_{t=0} = f_0 \in L^p(\R^3 \times \R^3).
\]
Suppose that
\[
\sup_{m \in \bbn} \|f^m\|_{L^\infty(\R^3 \times \R^3 \times (0,T))} + \sup_{m \in \bbn}\| (|v|^2 + |x|^2) f^m\|_{L^\infty(0,T;L^1(\R^3 \times \R^3))}  <\infty.
\]
Then, the sequence
\[ \left\{ \int_{\R^3} f^m \,dv \right\}_m \]
is relatively compact in $L^q(\R^3 \times (0,T))$ for any $q \in \left(1, \frac53\right)$.
\end{lemma}

We recall the uniform-in-$\e$ bound estimates obtained in Section \ref{ssec:uniform}:
\bq\label{est_uni}
\|f^{(\e)}\|_{L^\infty(0,T;L^p(\R^3\times\R^3))} + \|\rho^{(\e)}\|_{L^\infty(0,T;L^q(\R^3))}  \le C,
\eq
where $p \in [1,\infty]$, $q \in [1,\frac53]$, and $C>0$ is independent of $\e$. Thus, by weak compactness theory, we have the following weak convergence as $\e \to 0$ up to a subsequence:
\[
\begin{array}{lcll}
f^{(\e)} \rightharpoonup f  &\mbox{ in }&  L^\infty(0,T;L^p(\R^3\times\R^3)), & p\in[1,\infty],\\[2mm]
\displaystyle \rho^{(\e)} \rightharpoonup \rho & \mbox{ in } & L^\infty(0,T;L^p(\R^3)_, & p\in\lt[1,\frac53\rt],\\
\end{array}
\]
Let us write 
$$
G^\e:= \sigma (\nabla_v f^{(\e)} + vf^{(\e)}) - (  \nabla K^\e \star \rho^{(\e)}  )f^{(\e)}.
$$

We now claim that $G^\e \in L^p_{loc}(\R^3 \times \R^3 \times [0,T])$ for some $p \in (1,\infty)$ to apply the averaging lemma, Lemma \ref{lem_va}. A direct application of Calder\'on--Zygmund lemma gives
\[
\|(\nabla K^\e_{\rm M} \star \rho^{(\e)}  )f^{(\e)}\|_{L^q} \leq \|f^{(\e)}\|_{L^\infty}\| \nabla K_{\rm M} \star \rho^{(\e)}\|_{L^q}\leq C \|f^{(\e)}\|_{L^\infty}\|\rho^{(\e)}\|_{L^q}
\]
for $q < \frac53$. For the term with $K^\e_{\rm C}$, we use the following Hardy--Littlewood--Sobolev inequality 
\[
\| |\cdot|^{-\lambda} \star \rho\|_{L^q} \leq C\|\rho\|_{L^p} \quad \frac\lambda{3} = 1 - \frac1q + \frac1p
\]
with $\lambda = 2$, $q = \frac{15}{14}$ and $p = \frac53$ to estimate
\[
\|(\nabla K^\e_{\rm C} \star \rho^{(\e)}  )f^{(\e)}\|_{L^{\frac{15}{14}}} \leq   \|f^{(\e)}\|_{L^\infty} \||\cdot|^{-2} \star \rho^{(\e)} \|_{L^{\frac{15}{14}}} \leq C\|f^{(\e)}\|_{L^\infty} \|\rho^{(\e)}\|_{L^{\frac53}},
\]
where $C>0$ is independent of $\e$. We also easily find from Proposition \ref{prop_reg} that for any $q < 2$
\begin{align*}
\intrr |\nabla_v f^{(\e)} + vf^{(\e)}|^q\,dxdv &= \intrr (f^{(\e)})^{\frac{q}2} (f^{(\e)})^{-\frac{q}2}|\nabla_v f^{(\e)} + vf^{(\e)}|^q\,dxdv\cr
&\leq \lt( \intrr\frac{1}{f^{(\e)}} |\nabla_v f^{(\e)} + vf^{(\e)}|^2  \,dxdv\rt)^{\frac{q}2}\|f^{(\e)}\|_{L^{\frac{q}{2-q}}}^{\frac{q}{2}}.
\end{align*}
Combining all of the above estimates and \eqref{est_uni} yields $G^\e\in L_{loc}^{\frac{15}{14}}(\R^3\times\R^3 \times (0,T))$.
We then apply Lemma \ref{lem_va} to obtain, for $p \in (1, \frac53)$,
\bq\label{st_conv}
\begin{array}{lcl}
\displaystyle \rho^{(\e)} \to \rho & \mbox{ in } & L^p(\R^3\times (0,T)) \ \mbox{ and a.e.},\\
\end{array}
\eq
as $\e \to 0$, up to a subsequence. \\

\begin{proof}[Proof of Theorem \ref{main_thm}]
We are now in a position to show that the limit function $f$ satisfies the system \eqref{eq:VR} in the distributional sense. We shall show that 
\[
(\nabla K^\e \star \rho^{(\e)})f^{(\e)} \to (\nabla K \star \rho)f,
\] in the sense of distributions; the other terms are linear and therefore easily handled.

We choose $\Psi\in\mathcal{C}_c^\infty(\R^3 \times \R^3 \times [0,T])$ and get
\begin{align*}
\int_0^t& \intrr \left[(\nabla K^\e \star \rho^{(\e)})f^{(\e)} - (\nabla K \star \rho)f\right]\Psi\,dxdvds\\
&= \int_0^t \intr (\nabla (K^\e - K)\star \rho )\rho_\Psi \,dxds\\
&\quad+\int_0^t \intr \nabla K^\e \star (\rho^{(\e)} -\rho)\rho_\Psi^{(\e)}\,dxds + \int_0^t \intr (\nabla K^\e \star\rho)(\rho_\Psi^{(\e)}-\rho_\Psi)\,dxds\\
&=: \mbox{(I)}^\e + \mbox{(II)}^\e + \mbox{(III)}^\e.
\end{align*}
Here we denoted by $\rho_\Psi := \intr f \Psi\,dv$ and $\rho_\Psi^{(\e)} :=  \intr f^{(\e)} \Psi\,dv$.

Note that thanks to the uniform-in-$\e$ estimate for $f^{(\e)}$ in $L^\infty(\R^3\times\R^3\times(0,T))$ and the compact support of $\Psi$, we find
\bq\label{est_rho}
\rho_\Psi, \ \rho_\Psi^{(\e)} \in L^p(0,T; L^q(\R^3))
\eq
for any $p,q \in [1,\infty]$ uniformly in $\e$. \\

$\bullet$ Estimate of (I)$^\e$:  Note that $|(\nabla K^\e \star \rho) \rho_\Psi| \leq |\nabla K \star \rho| |\rho_{\Psi}|$ and $(\nabla K^\e \star \rho) \rho_\Psi \to (\nabla K\star \rho) \rho_\Psi$ pointwise $\e \to 0$. We use the Calder\'on--Zygmund lemma to get
\[
\int_0^t \intr |\nabla K_{\rm M} \star \rho| |\rho_\Psi | \,dxds \leq C\|\rho\|_{L^p(\R^3 \times (0,T))}\|\rho_\Psi\|_{L^{p'}(\R^3 \times (0,T))}
\]
for any $p \in (1, \frac53)$, where $p'$ is the H\"older conjugate of $p$. By \eqref{HLS}, we also observe 
\begin{align*}
\int_0^t \intr (|\nabla K_{\rm C}| \star \rho) |\rho_\Psi | \,dxds &\leq \int_0^t \intrr \rho(x) \frac{1}{|x-y|^2} |\rho_\Psi|(y)\,dxdyds \cr
&\leq C\|\rho\|_{L^p(\R^3 \times (0,T))} \|\rho_\Psi\|_{L^{p'}(0,T;L^q(\R^3))},
\end{align*}
where $p \in (\frac34,\frac53)$ and $q$ is given by $\frac1q = \frac43 - \frac1p > 0$, and $p'$ is the H\"older conjugate of $p$. Then, by the dominated convergence theorem, we have $\mbox{(I)}^\e \to 0$ as $\e \to 0$. \\

$\bullet$ Estimate of (II)$^\e$: Similarly as in the estimate of (I)$^\e$, we use the Calder\'on--Zygmund lemma to find
\[
\lt|\int_0^t \intr \nabla K^\e_{\rm M} \star (\rho^{(\e)} -\rho)\rho_\Psi^{(\e)}\,dxds\rt| \leq C\|\rho^{(\e)} -\rho\|_{L^p(\R^3 \times (0,T))}\|\rho_\Psi^{(\e)}\|_{L^{p'}(\R^3 \times (0,T))},
\]
for any $p \in (1, \frac53)$, where $C>0$ is independent of $\e>0$. This together with \eqref{est_rho} implies
\[
\int_0^t \intr \nabla K^\e_{\rm M} \star (\rho^{(\e)} -\rho)\rho_\Psi^{(\e)}\,dxds \to 0
\]
as $\e\to 0$. We also find for $p \in (\frac34, \frac53)$
\begin{align*}
\lt|\int_0^t \intr \nabla K^\e_{\rm C} \star (\rho^{(\e)} -\rho)\rho_\Psi^{(\e)}\,dxds\rt| &\leq  \int_0^t \intrr |\rho^{(\e)} -\rho|(x) \frac{1}{|x-y|^2} |\rho^{(\e)}_\Psi|(y)\,dxdyds\cr
& \leq C\|\rho^{(\e)} -\rho\|_{L^p(\R^3 \times (0,T))} \|\rho^{(\e)}_\Psi\|_{L^{p'}(0,T;L^q(\R^3))},
\end{align*}
where $q$ is given by $\frac1q = \frac43 - \frac1p > 0$.
Then by \eqref{st_conv}, we have $\mbox{(II)}^\e \to 0$ as $\e \to 0$.\\

$\bullet$ Estimate of (III)$^\e$: Similarly as before, we observe
\[
(\nabla K^\e \star \rho) \Psi \in L^1(0,T;L^\frac{15}{14}(\R^3))
\]
uniformly in $\e$. Thus, due to the weak convergence of $f^{(\e)}$, we get $\mbox{(III)}^\e \to 0$ as $\e \to 0$. \newline

Therefore, we conclude that $f$ is a weak solution to \eqref{eq:VR}. 
\end{proof}

\subsection*{Acknowledgements}
YPC has been supported by NRF grant (No. 2017R1C1B2012918) and Yonsei University Research Fund of 2021-22-0301. IJJ has been supported by the New Faculty Startup Fund from Seoul National University and the National Research Foundation of Korea grant (No. 2019R1F1A1058486).

\subsection*{Conflict of Interest}
We state that there is no conflict of interest.

\end{document}